%% file: circular_vortex_sheets_instability7-31.tex
\documentclass[11pt,twoside]{article}
\usepackage[hmargin=1.25in,vmargin=1.25in, a4paper, centering]{geometry}
\usepackage[nottoc]{tocbibind}

\usepackage[biblabel,nosort,nocompress]{cite}

\usepackage{graphicx}
\usepackage{amsmath,amsthm}
\usepackage{amssymb}
\usepackage{bbm}
\usepackage{upgreek}
\usepackage{biolinum}
\usepackage{garamondlibre}

\usepackage[notext,noDcommand,partialup]{kpfonts}
\usepackage{courierten}

\usepackage{physics}

\usepackage{graphicx}
\usepackage{mathtools,slashed}
\usepackage{cases}
\usepackage{xfrac}

\usepackage[dvipsnames]{xcolor}
\usepackage[colorlinks=true, pdfstartview=FitV, linkcolor=RoyalBlue,citecolor=ForestGreen, urlcolor=BurntOrange]{hyperref}

\usepackage{tikz}
\usetikzlibrary{hobby}
\usepackage{fancyhdr}
\renewcommand\footnotemark{}
\usepackage[sf,sl,outermarks]{titlesec}
\titleformat{\section}
{\Large\bfseries\sffamily}{\filcenter\Large{\thesection.}}{1em}{\filcenter}
\titleformat{\subsection}
{\large\bfseries\sffamily}{\thesubsection.}{1em}{}
\titleformat{\subsubsection}[runin]
{\normalsize\bfseries\itshape\sffamily}{{\normalfont\bfseries\itshape\S}\thesubsubsection.}{0.5em}{}[.\hspace*{0.5ex}]
\titleformat{\paragraph}[runin]
{\normalsize\bfseries\itshape\sffamily}{{\normalfont\bfseries\itshape\S}\theparagraph.}{0.5em}{}[.\hspace*{0.5ex}]

\usepackage{abstract}

\input{mathsymb.tex}

\begin{document}
	\title{%
		\LARGE \bfseries \sffamily
		Interactions between Wind and Water Waves near Circular Flows%
		\footnote{{\bf\sffamily Date: }\today.}
	}
	\author{%
		\scshape Changfeng GUI
		\thanks{%
			\textbf{\sffamily Changfeng GUI: } Department of Mathematics, Faculty of Science and Technology, University of Macau, Taipa, Macao SAR, China.
			\textit{E-mail}: {\color{BurntOrange} \ttfamily changfenggui@um.edu.mo}
		}
		\and
		\scshape Sicheng LIU
		\thanks{%
			\textbf{\sffamily Sicheng LIU: } Department of Mathematics, Faculty of Science and Technology, University of Macau, Taipa, Macao SAR, China.
			\textit{E-mail}: {\color{BurntOrange} \texttt{scliu@link.cuhk.edu.hk}}
		}
	}
	\date{} 
	\maketitle
	\begin{abstract}  
		This manuscript concerns the dynamical interactions between wind and water waves, which are characterized through two-phase free interface problems for the Euler equations. We provide a comprehensive derivation on the linearized problems of general two-phase flows. Then, we study the instability issues of perturbing waves around circular steady solutions, and we demonstrate a semi-circle result on the possible locations of unstable modes. We also present necessary conditions and sufficient ones for the instability of wind-perturbing water waves near Taylor-Couette flows.
	\end{abstract}

	{\sffamily\tableofcontents}
	
	\addtocontents{toc}{\protect\setcounter{tocdepth}{2}}

	\section{Introduction}
	We consider the interactions between wind and water waves. If the atmosphere occupies a sufficiently large domain, the compression of air is sometimes negligible. Specifically, when one focuses on the interactions between wind and water near the interface separating them, it is natural to assume that the dynamics can be characterized through the two-phase free interface problems for incompressible Euler equations
	\begin{equation}\label{euler}
		\begin{cases*}\displaystyle
		\varrho_{\pm}\qty\big[\pd_t \vv_{\pm} + (\vv_{\pm}\vdot\grad)\vv_{\pm}]+\grad p_{\pm} = \vb{0} &in $\cU_t^{\pm} $,\\
			\div\vv_{\pm} = 0 &in $\cU_t^\pm$,
		\end{cases*}
	\end{equation}
	together with boundary conditions
	\begin{equation}\label{bc}
		\begin{cases*}
			p_+ - p_- = \alpha\varkappa &on $\Gmt$,\\
			\vv_+ \vdot \vn = \vv_- \vdot \vn = \cV &on $\Gmt$,
		\end{cases*}
	\end{equation}
	where $\varrho_\pm, \vv_\pm, p_\pm$ are respectively the density, velocity field, and pressure of the two fluids, $\cU_t^\pm$ are the moving fluid domains, $\Gmt$ is the free interface separating two fluids, $\vn$ is the unit normal vector field of $\Gmt$ (which is assumed to be the outer normal of $\Gmt \subset \pd\cU_t^+$), $\cV$ is the normal speed of $\Gmt$ in the direction of $\vn$, $\varkappa$ is the mean curvature of $\Gmt$ with respect to $\vn$, and $\alpha \ge 0$ is a constant measuring the surface tension.
	
	The first equation in \eqref{euler} describes the conservation of momentum of two fluids, the second one is the incompressibility condition. The first boundary condition in \eqref{bc} follows from the balance of momenta across the free interface, and the second one indicates that the free interface evolves with the fluids. One can refer to \cite[Ch. 7]{LL87book} for more detailed derivations and discussions on the physical backgrounds of interfacial waves.
	
	For the sake of simplicity, we restrict our attention to the 2D problems.
	
	\subsection{Backgrounds and Related Works}
	
	The mathematical study of wind–water interactions, which are formulated as two-phase free interface problems in hydrodynamics, can reveal how the profiles of wind and water facilitate energy transfer and trigger wave instabilities. Wind, as a driving force, introduces minute perturbations on the water surface tangentially.  Under suitable conditions, such as marked differences in velocity, these disturbances can amplify into observable waves, which is the well-known Kelvin–Helmholtz instability. These problems not only characterize the fundamental generation and evolution process of ocean surface waves but also provide a robust theoretical basis for forecasting extreme weather events, designing marine structures, and harnessing wave energy. The mathematical results quantitatively capture these complex processes, offering insights into how wind profiles, surface tension, and fluid density influence wave developments and break-up. In essence, the study of wind–water interfaces deepens our understanding of transfer mechanisms of natural energy, fueling both the exploration of fundamental physical processes and the advancement of innovative engineering solutions.
	
	Local well-posedness of free interface problems \eqref{euler}-\eqref{bc} with $\alpha > 0$ (i.e., capillary vortex sheets) in all space dimensions has been established in standard Sobolev spaces, one may refer to \cite{CCS08CPAM} and \cite{SZ11ARMA} for the results using different approaches. In the absence of surface tension, the vortex sheet problems are ill-posed for the linearized issues (cf. \cite{BHL93CPAM}) and the nonlinear scenarios (see \cite{Ebin88CPDE}).
	
	As particular cases, there is a large class of exact (steady) solutions to \eqref{euler}-\eqref{bc} taking the form in Cartesian coordinates:
	\begin{equation*}
		\vv(t, x) = \qty(v^1(x_2), 0)^T \qc \Gmt = \{x_2 = 0\},
	\end{equation*}
	which are called \emph{shear flows}. Such models characterize the simplest cases for air-ocean interactions. If the ocean is assumed to be quiescent and the air velocity is uniform, this becomes the classical Kelvin-Helmholtz model, which is linearly unstable if $\alpha = 0$ and the air velocity is non-zero. More detailed discussions on the Kelvin-Helmholtz instability can be found in \cite{DR04book}. For more general wind profiles (still under the shear flow settings), {\sc Miles} studied the wind generated capillary gravity water waves in a series of works  \cite{Miles57JFM,Miles59JFM,Miles59JFM'}, revealing the phenomena that instabilities can be caused by some particular wind profiles, under the hypothesis that the density of air is sufficiently small. Rigorous mathematical justifications of {\sc Miles'} criterion were demonstrated by {\sc B\"{u}hler, Shatah, Walsh, and Zeng} in \cite{BuhlerSWZ2016}, which also includes very detailed surveys on the air-ocean interface problems. A very recent research by {\sc Liu} \cite{Liu2024} addresses careful analysis on the spectrum for the linearized two-phase problems at the shear flows. For more discussions on the stability and instability of surface waves around shear flows, see \cite{HL08CMP,HL08err,RR13JFM,BR13Angew,LZ21arxiv,LZ25CMP}.
	
	However, solutions close to shear flows can only describe surface waves for the graph-type free interfaces. In many realistic scenarios, the free interface cannot be fully represented by a graph (e.g., liquid drops and water columns). Thus, it would be natural to study the free boundary problems without graph assumptions. Although the local well-posedness is available, the stability analysis is still in its infancy. Unlike the shear flows with flat interface, to the extent of our knowledge, the stability or instability analysis for free boundary problems with non-flat interface is still open. This motivates us to consider surface waves around the circular flows in an annular region or a disk. Concerning the stability for fixed-boundary Euler equations around circular flows, one can refer to \cite{Zil17JDE}, the survey \cite{Gal19arxiv}, and the references therein.

	\subsection{Summary of Results}
	
	\subsubsection{Linearized problems}
	
	In \S\ref{sec lin eqn}, we derive the linearized systems for the two-phase free interface problems in a rigorous manner. Particularly, the interface is not presumed to be a graph, and the background velocity can be an arbitrary solenoidal vector fields satisfying \eqref{bc}. The main issue is to linearize the boundary conditions \eqref{bc} on the free interface. To the extent of our knowledge, this is the first rigorous characterization on the linearized problems for general two-phase flows. (One can refer to \cite{BuhlerSWZ2016} for the results under graph assumptions through a different approach.) Particularly, the jump condition on the linearized velocity fields are intrinsically contained in the evolution equation for the linearized normal speed. Moreover, our arguments can be easily extended to higher dimensional scenarios.
	
	\subsubsection{Perturbations of circular flows}\label{sec circular background}
	After \S\ref{sec lin eqn}, we mainly focus on the instability issues of water waves around circular flows. In terms of polar coordinates $(r, \theta)$, one may express a vector in the form
	\begin{equation}
		\vv = v^r\er + v^\theta\etheta.
	\end{equation}
	We consider the axisymmetric circular background flows, i.e.,
	\begin{equation}\label{def V}
		\vV_\pm = rW_\pm(r) \etheta,
	\end{equation}
	where $W_\pm(r)$ represent the angular velocities. Particularly, the vorticities can be expressed as
	\begin{equation}\label{def Omega}
		\Omega_\pm = \Curl \vV_\pm = \frac{1}{r}\qty[\pd_r\qty(rV_\pm^\theta) - \pd_\theta V_\pm^r] = \frac{\pd_r \qty(r^2 W_\pm)}{r} = 2W_\pm + r\pd_rW_\pm.
	\end{equation} 
	Consider the perturbation of the stationary flows
	\begin{equation}
		\begin{cases*}
			\varrho_\pm (\vV_\pm \vdot \grad)\vV_\pm + \grad P_{\pm} = \vb{0} &in $\cUs^\pm$, \\
			\div \vV_\pm = 0 &in $\cUs^\pm$,
		\end{cases*}
	\end{equation}
	for which (here $\Rin$ and $\Rout$ are two constants).
	\begin{equation}\label{def Gmt}
		\cUs^+ = \{\Rin < r < 1\} \qc \cUs^- = \{ 1 < r < \Rout \}, \qand \Gms = \qty{r = 1}.
	\end{equation}
	Assume further that on the fixed boundaries, there hold the slip boundary conditions:
	\begin{equation}
		\vV_+ \vdot \vn_{\text{in}} = 0 \text{ on } \{r = \Rin\} \qand \vV_- \vdot \vn_{\text{out}} = 0 \text{ on } \{r = \Rout\}.
	\end{equation}
	We also consider the scenario where the inner fluid region is the unit disk, i.e.
	\begin{equation}
		\cUs^+ = \{r < 1\},
	\end{equation}
	in which case, there would be no boundary condition imposed on the origin.
	
	In \S\ref{sec lin cir}, we derive the equations and dispersive relations for perturbing waves around circular background flows. After that, we analyze the (in-)stability issues for some typical models in \S\ref{sec ex}, which, particularly, indicate the stabilization effect of capillary forces. In \S\ref{sec semi-circle}, we demonstrate a semi-circle type result on the location of possible unstable phase velocities in the style of {\sc Howard} \cite{Howard1961} (see also \cite{Liu2024}), which only depend on the extreme values of background angular velocities.
	
	\subsubsection{Instability near Taylor-Couette type water flows}
	In \S\S\ref{sec TC flow}-\ref{sec CL}, under the inspiration of \cite{BuhlerSWZ2016}, we study the instability issues of Taylor-Couette water flows perturbed by circular wind with small densities. More precisely, we first consider the limiting water-vacuum free interface problems around Taylor-Couette-type background water flows, and we assume that for a fixed wave number, the perturbing wave is linearly stable with two distinct real phase velocities. We then demonstrate that, if the wind profile is sufficiently regular, the only possible unstable mode can only bifurcate from the critical layers (i.e., locations at which the wind takes the value of the prescribed phase velocity) for the wind, which can be regarded as a necessary condition for the wind-generated instability. Conversely, if at least one of these two phase velocities is a regular value of the angular velocity of the wind profile, then, under some additional assumptions on the sign of the derivative of wind vorticity at critical layers, the wind-perturbed water wave would be linearly unstable, which is a sufficient condition for the wind-generated instability. In \S\ref{sec other instability}, we provide an example existing instabilities for Lipschitz wind profiles, for which the critical layer and the support of the derivative of wind vorticity are disjoint with a positive bound, indicating that irregular wind profiles can lead to different instability mechanisms.
	
	\section{Linear Problems}
	
	\subsection{Linearization of the Free Interface Problems}\label{sec lin eqn}
	In this section, we consider the linearization of two-phase free interface problems \eqref{euler}-\eqref{bc}. Here we remark that the derivations in \S\ref{sec lin eqn} do not depend on the formulations of background solutions. 
	
	Denote by $\grad$ the covariant derivative in $\R^2$. Suppose that $\varrho_{\pm}$ and $\alpha$ are generic constants, and $(\vv_\pm, p_\pm, \cU_t^\pm)_\beta$ are a family of solutions to \eqref{euler}-\eqref{bc} parameterized by $\beta$. Let
	\begin{equation}
		\vbu \coloneqq (\pd_\beta\vv)_{\restriction \beta = 0} \qand q \coloneqq (\pd_\beta p)_{\restriction \beta = 0}
	\end{equation}
	be the linearized variables of $\vv$ and $p$, respectively. Then, taking variational derivatives of \eqref{euler} with respect to $\beta$ and evaluate at $\beta = 0$ yield the following linear equations for $(\vbu_\pm, q_\pm)$:
	\begin{equation}\label{lin eqn}
		\begin{cases*}
			\varrho_{\pm} \qty(\pd_t \vbu_\pm + \grad_{\vv_\pm}\vbu_\pm + \grad_{\vbu_\pm}\vv_\pm) + \grad q_{\pm} = \vb{0} &in $\cU_{t}^\pm$, \\
			\div \vbu_\pm = 0 &in $\cU_{t}^\pm$,
		\end{cases*}
	\end{equation}
	where $(\vbu_\pm, q_\pm)$ are interpreted as functions defined in $\cU_{t}^\pm$ for each fixed time moment.
	
	Next, we linearize the boundary conditions \eqref{bc}. Let $\vpsi$ be the variational velocity field of $\Gamma_{t, \beta}$ with respect to the parameter $\beta$, which means that $\vpsi$ is a vector field defined on $\Gamma_{t, \beta}$. For the simplicity of notations, we denote by 
	\begin{equation}
		\Dt \coloneqq \pd_t + \grad_\vv
	\end{equation}
	the material derivative along the fluid particle path, and
	\begin{equation}
		\Dbt \coloneqq \pd_\beta + \grad_{\vpsi}
	\end{equation}
	the material derivative along the variational trajectory. Furthermore, let $(\vtau, \vn)$ be the frame of $\Gmt$, where $\vtau$ is the unit tangent filed and $\vn$ is the unit outer normal of $\Gmt\subset\cU_t^+$. The orientation is prescribed so that $\vtau$ is the counterclockwise rotation of $\vn$ with angle $\pi/2$ (i.e., for the unit circle, there holds $\vtau = \etheta$ and $\vn=\er$). Then, it is clear that
	\begin{equation}\label{def ka}
		\grad_\vtau\vtau = - \varkappa\vn \qand \grad_\vtau\vn = \varkappa\vtau.
	\end{equation}
	Due to the unit length of $\vtau$, one can calculate from coordinate expressions that (cf. \cite[\S3.1]{Liu-Xin2023})
	\begin{equation}\label{Dt tau n}
		\Dt\vtau = (\grad_\vtau\vv \vdot\vn)\vn \qand \Dt\vn = - (\grad_\vtau\vv\vdot\vn)\vtau.
	\end{equation}
	Particularly, it follows that
	\begin{equation}\label{Dt ka}
		\Dt\varkappa = \Dt(\grad_\vtau\vn\vdot\vn) = -\vn\vdot(\grad_\vtau\grad_\vtau\vv) - 2\varkappa(\grad_\vtau\vv\vdot\vtau).
	\end{equation}
	It is obvious that the formulae \eqref{Dt tau n}-\eqref{Dt ka} also hold for $\Dbt$ with $\vv$ replaced by $\vpsi$, as they are purely geometrical relations. For the sake of convenience, we denote by
	\begin{equation}\label{def psi}
		\vpsi = \psi^\top \vtau + \psi^\perp\vn,
	\end{equation}
	where $\psi^\top$ and $\psi^\perp$ are both functions defined on $\Gamma_{t, \beta}$.
	
	The first boundary relation in \eqref{bc} reads that
	\begin{equation*}
		p_+ - p_- = \alpha\varkappa \qq{on} \Gmt.
	\end{equation*}
	Taking the material derivative with respect to $\beta$ implies that
	\begin{equation*}
		\Dbt(p_+ - p_-) \equiv \alpha \Dbt\varkappa \qq{on} \Gamma_{t, \beta},
	\end{equation*}
	which yields
	\begin{equation}
		(q_+ - q_-) + \psi^\perp (\grad_\vn p_+ - \grad_\vn p_-) + \psi^\top \alpha \grad_\vtau\varkappa = -\alpha \qty[\vn\vdot(\grad_\vtau\grad_\vtau\vpsi) + 2\varkappa(\grad_\vtau\vpsi\vdot\vtau)].
	\end{equation}
	It follows from \eqref{def ka} and \eqref{def psi} that
	\begin{equation}
		q_+ - q_- = -\alpha\grad_\vtau\grad_\vtau\psi^\perp - \qty[\alpha\varkappa^2 + (\grad_\vn p_+ - \grad_\vn p_-)]\psi^\perp  \qq{on} \Gmt.
	\end{equation}
	
	The second boundary condition in \eqref{bc} means that the fluid particles lying on the free interface will never leave. Taking a $C^2$ defining function $\Psi$ of $\Gamma_{t, \beta}$, for which
	\begin{equation*}
		\Gamma_{t, \beta} = \qty{\Psi(t, \beta, x) = 0} \qand \det(\grad_x \Psi) > 0 \text{ in a neighborhood of } \Gamma_{t, \beta}.
	\end{equation*}
	Then, it follows that
	\begin{equation}
		\Dt\Psi \equiv 0 \qand \Dbt\Psi \equiv 0 \qq{on} \Gamma_{t, \beta}.
	\end{equation}
	Through extending $\vpsi$ into a neighborhood of $\Gamma_{t, \beta}$ in an appropriate sense, it is legitimate to calculate that
	\begin{equation*}
		\comm{\Dt}{\Dbt} \coloneqq \Dt\Dbt - \Dbt\Dt = \grad_{\Dt\vpsi-\Dbt\vv}.
	\end{equation*}
	On the other hand, it is obvious that
	\begin{equation*}
		\comm{\Dt}{\Dbt}\Psi = \grad_{\Dt\vpsi-\Dbt\vv} \Psi = 0 \qq{on} \Gamma_{t, \beta},
	\end{equation*}
	which implies
	\begin{equation*}
		\qty(\Dt\vpsi-\Dbt\vv) \vdot \vn = 0 \qq{on} \Gamma_{t, \beta}.
	\end{equation*}
	Particularly, there holds
	\begin{equation}\label{Dt psi}
		\Dt\psi^\perp = \vbu\vdot\vn + \psi^\perp (\grad_\vn \vv \vdot \vn) \qq{on} \Gamma_{t}.
	\end{equation}
	Indeed, \eqref{Dt psi} implies
	\begin{equation*}
		\grad_{\vv_+ - \vv_-}\psi^\perp = (\vbu_+ - \vbu_-) \vdot \vn + \psi^\perp (\grad_\vn\vv_+ \vdot\vn - \grad_\vn \vv_- \vdot\vn).
	\end{equation*}
	Since $\div\vv_\pm = 0$, it holds that
	\begin{equation*}
		0 = \div\vv_\pm = \grad_\vtau\vv_\pm \vdot \vtau + \grad_\vn \vv_\pm \vdot \vn,
	\end{equation*}
	which yields
	\begin{equation}\label{jump u}
		(\vbu_+ - \vbu_-)\vdot\vn = \grad_{\vv_+ - \vv_-}\psi^\perp + \psi^{\perp}\vtau\vdot\grad_\vtau(\vv_+ - \vv_-).
	\end{equation}
	Specifically, when $\psi^\perp \neq 0$ and $(\vv_+ - \vv_-)\vdot\vn \neq 0$, the linearized velocity fields $\vbu_\pm$ would have jumps in the normal direction among the free interface.
	
	On the other hand, it follows from \eqref{bc}\textsubscript{2} that
	\begin{equation*}
		(\vv_+ - \vv_-) \vdot \vn = 0 \qq{on} \Gmt,
	\end{equation*}
	which, after applying $\Dbt$, leads to
	\begin{equation*}
		(\vbu_+ - \vbu_-)\vdot\vn + \grad_{\vpsi}(\vv_+ - \vv_-) \vdot\vn + (\vv_+ - \vv_-)\vdot\Dbt\vn = 0.
	\end{equation*}
	Thus, \eqref{Dt tau n} and \eqref{def psi} imply that
	\begin{equation*}
		(\vbu_+ - \vbu_-)\vdot\vn = \grad_\vtau\qty[\psi^\perp \vtau\vdot(\vv_+ - \vv_-)] \qq{on} \Gamma_{t},
	\end{equation*}
	which coincide with the previously derived jump condition \eqref{jump u}.
	
	In summary, the boundary conditions for the linearized problems for \eqref{euler}-\eqref{bc} can be written as
	\begin{equation}\label{lin bc}
		\begin{cases*}
			q_+ - q_- = -\alpha\grad_\vtau\grad_\vtau\psi^\perp - \qty[\alpha\varkappa^2 + (\grad_\vn p_+ - \grad_\vn p_-)]\psi^\perp &on $\Gamma_{t}$, \\
			\pd_t\psi^\perp + \grad_{\vv_\pm} \psi^\perp = \vbu_\pm \vdot \vn + (\grad_\vn\vv_\pm\vdot\vn)\psi^\perp &on $\Gamma_{t}$.
		\end{cases*}
	\end{equation}
	
	\subsection{Linear Evolutions around Circular Flows}\label{sec lin cir}
	From now on, we mainly focus on the eigenvalues and eigenfunctions of the linearized operator among the background flows introduced in \S\ref{sec circular background}. Namely, we consider the perturbations of circular flows with profile (recall that the stationary free interface is assumed to be the unit circle):
	\begin{equation}\label{ptb profile}
		\vbu(t, r, \theta) = \underline{\vbu}(r) e^{\lambda t} e^{ik\theta} \qc q(t, r, \theta) = \underline{q}(r) e^{\lambda t} e^{ik\theta}, \qand \psi^\perp(t, \theta) = \underline{\psi} e^{\lambda t} e^{ik\theta},
	\end{equation}
	where $k \in \Z\setminus\{0\}$, $\lambda\in\C$, and $\underline{\psi} \in \C\setminus\{0\}$ are constants. Here $\underline{\psi}$ is assumed to be non-trivial since we are considering surface waves.
	
	Plugging the profiles \eqref{def V}, \eqref{def Gmt}, and \eqref{ptb profile} into the linear equations \eqref{lin eqn}, it follows that
	\begin{equation}\label{lin eqn u}
		\begin{cases*}
			\lambda u^r - W u^\theta + ikW u^r - Wu^\theta + \varrho^{-1}\pd_r q = 0, \\
			\lambda u^\theta + u^r \pd_r \qty(rW) + ikW u^\theta + Wu^r + \varrho^{-1} r^{-1} ik q = 0, \\
			\pd_r\qty(r u^r) + ik u^\theta = 0.
		\end{cases*}
	\end{equation}
	Particularly, one can express $u^\theta$ in terms of $u^r$ through \eqref{lin eqn u}\textsubscript{3} and arrive at the equation involving only $u^r$:
	\begin{equation}
		i k^{-1}(\lambda+ikW)\pd_r\qty\big[r\pd_r\qty(ru^r)] + \qty\big[-ik(\lambda+ikW) + r\pd_r(2W+r\pd_rW)] u^r = 0.
	\end{equation}
	Denote by
	\begin{equation}
		c \coloneqq \frac{i \lambda}{k} = \cR + i\cI \iff \lambda = - ikc,
	\end{equation}
	where $c \in \C$ and $\cR, \cI \in \R$. Then, it holds that
	\begin{equation}
		-\pd_r \qty\big[r\pd_r\qty(ru^r)] + \qty[k^2 + \frac{r\pd_r\Omega}{W-c}] u^r = 0,
	\end{equation}
	where $\Omega$ is the vorticity given by \eqref{def Omega}.
	
	For the simplicity of expressions, we introduce the new variable:
	\begin{equation}
		s \coloneqq \log r \qfor r > 0.
	\end{equation}
	Thus, it is obvious that
	\begin{equation}
		\dv{s} = \qty(r\dv{r})_{\restriction{r=e^s}}.
	\end{equation}
	For the simplicity of notations, we denote by
	\begin{equation}
		\dot{f} \coloneqq \dv{s} f.
	\end{equation}
	Define the new variables:
	\begin{equation}\label{def new variable}
		z (s) \coloneqq \qty(r\underline{u}^r)_{\restriction{r=e^s}} \qc w(s) \coloneqq W(e^s), \qand \varpi(s) \coloneqq \Omega(e^s) = 2w(s) + \dot{w}(s).
	\end{equation}
	Then, there holds
	\begin{equation}\label{ODE}
		-\ddot{z}(s) + \qty[k^2 + \frac{\dot{\varpi}(s)}{w(s)-c}]z(s) = 0,
	\end{equation}
	which takes the form of celebrated Rayleigh's stability equation. As $u^r$ also characterizes the (linearized) normal velocity on the fixed boundaries, it is natural to impose the boundary conditions
	\begin{equation}\label{z bc fixed}
		z_+(\log \Rin) = z_- (\log \Rout) = 0.
	\end{equation}
	When $\cU^+$ is assumed to be the unit disk, the boundary condition of $ru^r_+$ at the origin is simply the compatibility condition
	\begin{equation*}
		\qty\big(ru^r_+)_{\restriction{r=0}} = 0,
	\end{equation*}
	which is equivalent to
	\begin{equation}
		\lim_{s \to -\infty} z_+(s) \eqqcolon z_+(-\infty) = 0.
	\end{equation}
	
	Concerning the boundary conditions on the free interface, it follows from \eqref{lin bc}\textsubscript{1} and \eqref{lin eqn u}\textsubscript{2} that
	\begin{equation*}
		\begin{split}
			&r \qty{\varrho_-\qty[\lambda u^\theta_- + u^r_- \pd_r(rW_-) + ikW_- u^\theta_- + w_- u^r_-] - \varrho_+\qty[\lambda u^\theta_+ + u^r_+\pd_r(rW_+) + ik W_+ u^\theta_+ + W_+ u^r_+]} \\
			&\quad = ikr^{-2} \qty[\alpha \pd_\theta\pd_\theta\psi^\perp - \alpha r^2 \varkappa^2 \psi^\perp - r^2(\grad_\vn P_+ - \grad_\vn P_-) \psi^\perp] \qq{on} \Gms.
		\end{split}
	\end{equation*}
	Since the background velocity fields and free interface admit the profiles \eqref{def V} and \eqref{def Gmt} respectively, one obtains that
	\begin{equation*}
		\grad_\vn P_\pm = \pd_r P_\pm = \varrho_\pm rW_\pm^2 \qq{on} \Gms.
	\end{equation*}
	Particularly, by invoking the assumption that $r \equiv 1$ on $\Gms$, it follows from \eqref{lin eqn u}\textsubscript{3} that
	\begin{equation*}
		\begin{split}
			&\varrho_- \qty\big[(c - W_-)r\pd_r(ru^r_-)+ 2W_- ru^r_- + ru^r_- r\pd_rW_-] - \varrho_+ \qty\big[(c-W_+)r\pd_r(ru^r_+) + 2W_+ ru^r_+ + ru^r_+ r\pd_rW_+]  \\
			&\quad = ik\psi^\perp \qty[\alpha(k^2 -1) - r(\varrho_+ W_+^2 - \varrho_- W_-^2)] \qq{on} \{r = 1\}.
		\end{split}
	\end{equation*}
	In terms of the $s$-coordinates and new variables defined through \eqref{def new variable}, the above relation can be written as
	\begin{equation}\label{pre dispersive}
		\begin{split}
			&ik\underline{\psi}\qty[\alpha(k^2 -1) - \varrho_+ w_+^2(0) + \varrho_- w_-^2(0)] \\
			&\quad = \varrho_-\qty\big{[c-w_-(0)]\dot{z}_-(0) + 2\varpi_-(0)z_-(0)}  - \varrho_+\qty\big{[c-w_+(0)]\dot{z}_+(0) + \varpi_+(0)z_+(0)}.
		\end{split}
	\end{equation}
	As $z_\pm$ admits homogeneous boundary conditions on the fixed boundaries and $\underline{\psi}$ is assumed to be a non-zero constant, one can simply take
	\begin{equation}\label{normalize psi}
		ik\underline{\psi} = 1.
	\end{equation}
	
	Next, the evolution equation for $\psi^\perp$ \eqref{lin bc}\textsubscript{2} and the profiles \eqref{def V}, \eqref{def Gmt}, and \eqref{ptb profile} yield that
	\begin{equation*}
		ik\underline{\psi} (W_\pm - c) = u^r_\pm \qq{on} \Gms,
	\end{equation*}
	which, together with the normalization convention \eqref{normalize psi}, imply that
	\begin{equation}\label{z0 value}
		w_\pm(0) - c = z_\pm (0).
	\end{equation}
	Plugging \eqref{normalize psi}-\eqref{z0 value} into \eqref{pre dispersive}, one can obtain the dispersive relation:
	\begin{equation}\label{disp z}
		\begin{split}
			&\alpha(k^2 -1) - \varrho_+ w_+^2(0) + \varrho_- w_-^2(0) \\
			&\quad = \varrho_+\qty\big{\qty[w_+(0)-c]\dot{z}_+(0) - \varpi_+(0)z_+(0)} - \varrho_-\qty\big{\qty[w_-(0)-c]\dot{z}_-(0)-\varpi_-(0)z_-(0)}
		\end{split}
	\end{equation}
	For the sake of technical convenience, we define the functions $\zeta_\pm$ through (which is legitimate unless the boundary value problems for $z_\pm$ admit merely trivial solutions):
	\begin{equation}
		z_\pm (s) \equiv \qty\big[w_\pm(0)-c] \cdot \zeta_\pm (s).
	\end{equation}
	Then, it is obvious that $\zeta_\pm$ solve the boundary value problems:
	\begin{equation}\label{BP}
		\begin{cases*}\displaystyle
			-\dv[2]{s} \qty\big(\zeta_\pm)(s) + \qty(k^2 + \frac{\dot{\varpi}_\pm (s)}{w_\pm(s) - c}) \zeta_\pm(s) = 0, \\
			\zeta_\pm(0) = 1 \qc \zeta_-(\log\Rout) = 0 \qc \zeta_+(\log\Rin) = 0,
		\end{cases*}
	\end{equation}
	together with the dispersive relation:
	\begin{equation}\label{dispersion}
		\begin{split}
			&\alpha(k^2 -1) - \varrho_+ w_+^2(0) + \varrho_- w_-^2(0) \\
			&\quad = \varrho_+\dot{\zeta}_+(0)\qty[w_+(0)-c]^2 - \varrho_+\varpi_+(0)\qty[w_+(0)-c] + \\
			&\qquad - \varrho_-\dot{\zeta}_-(0)\qty[w_-(0)-c]^2 + \varrho_-\varpi_-(0)\qty[w_-(0)-c].
		\end{split}
	\end{equation}	
	In summary, the dynamics of linear perturbations with profile \eqref{ptb profile} around the circular background flows \eqref{def V} and \eqref{def Gmt} can be characterized by the ODE boundary value problems \eqref{BP} together with the dispersive relation \eqref{dispersion}.
	
	\begin{defi}\label{def mode}
		A pair $(c, k)$ with $c \in \C$ and $k \in \Z\setminus\{0\}$ is called a (neutral) mode if the ODE boundary value problems \eqref{BP} are solvable, whose solutions also satisfy the dispersive relation \eqref{dispersion}. This mode is called unstable if $\cI \coloneqq \Im{c} > 0$.
	\end{defi}
	
	\subsection{Examples}\label{sec ex}
	
	Here we present some examples with simple background flows.
	
	\subsubsection{One-phase flows (water waves)}
	We first assume that $\varrho_- = 0$, then, the problem is reduced to classical (capillary) water wave issues.
	
	\begin{ex}[Violation of Taylor's sign condition]\label{ex 1}
		Suppose that $\alpha = 0$, $\varrho_+ = 1$, $w_+ \equiv 1$, and $\cUs^+$ is the unit disk. Then, the linear problem is the perturbation of free-boundary Euler equations without surface tension around constant vortices. The ODE in \eqref{BP} can be written as
		\begin{equation*}
		\dv[2]{s} (\zeta_+) = k^2 \zeta_+ \qfor s <0.
		\end{equation*}
		The boundary conditions are
		\begin{equation*}
			\zeta_+(0) = 1 \qand \zeta_+ (-\infty) = 0.
		\end{equation*}
		Therefore, the solution can be expressed as:
		\begin{equation*}
			\zeta_+(s) = \exp(\ak s) \qfor s \le 0.
		\end{equation*}
		In particular, one has
		\begin{equation*}
			\dot{\zeta}_+(0) = \ak.
		\end{equation*}
		Thus, the dispersive relation \eqref{dispersion} reads
		\begin{equation*}
			\qty[c-\qty(1-\frac{1}{\ak})]^2 = - \frac{1}{\ak}\qty(1-\frac{1}{\ak}),
		\end{equation*}
		which has two non-real roots if $\ak \ge 2$. Namely, the constant vortex is linearly unstable for all large wave numbers, which coincides with classical instability/ill-posedness results (see, for example, \cite{Ebin1987}).
	\end{ex}
	
	\begin{ex}[Stabilizing effect of capillary forces]
		Suppose that $\alpha > 0$, $\varrho_+ = 1$, $w_+ \equiv \text{const.} B$, and $\cUs^+$ is the unit disk. Then, the only difference to Example \ref{ex 1} is the existence of surface tension. The dispersive relation \eqref{dispersion} now reads
		\begin{equation*}
			\qty[c-B\qty(1-\frac{1}{\ak})]^2 = \frac{\ak-1}{k^2}\qty[\alpha \ak(\ak+1) - B^2].
		\end{equation*}
		Hence, there exists no unstable mode for large wave numbers. Particularly, the capillary water wave problem around constant vortices has no unstable modes provided that
		\begin{equation*}
			6\alpha \ge B^2.
		\end{equation*}
		This reflects the stabilizing effect of the surface tension.
	\end{ex}
	
	\begin{ex}[Surface waves around Taylor-Couette flows]\label{ex TC flow}
		Suppose that $0 < \Rin < 1$, $w_+(s) = Ae^{-2s} + B$ (here $A, B \in \R$ are both constants), which corresponds to the Taylor-Couette flow:
		\begin{equation*}
			\vV_+ = \qty(\frac{A}{r} + Br) \etheta.
		\end{equation*}
		In this case, $\varpi_+(s) \equiv 2B$, and $\zeta_+$ solves the boundary value problem:
		\begin{equation*}
			\begin{cases*}
				\dv[2]{s} (\zeta_+) = k^2 \zeta_+ \qfor \log\Rin < s < 0,\\
				\zeta_+(0) = 1 \qand \zeta_+(\log\Rin) = 0.
			\end{cases*}
		\end{equation*}
		It is standard to calculate that
		\begin{equation*}
			\zeta_+(s) = \frac{1}{1-\Rin^{2\ak}} \exp(\ak s) - \frac{\Rin^{2\ak}}{1-\Rin^{2\ak}}\exp(-\ak s) \qfor \log\Rin \le s \le 0.
		\end{equation*}
		Therefore, the dispersive relation \eqref{dispersion} can be rewritten as
		\begin{equation}\label{disp TC ww}
			\frac{1+\Rin^{2\ak}}{1-\Rin^{2\ak}}\cdot \ak\qty[c-\qty(A+B-\frac{1-\Rin^{2\ak}}{\ak(1+\Rin^{2\ak})}\cdot B)]^2 = \frac{\alpha}{\varrho_+}\qty(k^2-1) + \frac{1-\Rin^{2\ak}}{\ak(1+\Rin^{2\ak})}\cdot B^2 - (A+B)^2
		\end{equation}
		In particular, when $ (A+B) = 0$, the water wave around Taylor-Couette flows is linearly stable regardless the existence of surface tension. Indeed, $(A+B) = 0$ corresponds to the quiescent water-vacuum interface.
	\end{ex}
	
	\subsubsection{Two-phase flows (vortex sheets)}
	For the simplicity of notations, we assume that $\varrho_+ > 0$ and denote by
	\begin{equation*}
		\varepsilon \coloneqq \frac{\varrho_-}{\varrho_+}.
	\end{equation*}
	\begin{ex}[Interactions between two Taylor-Couette flows]
	Suppose that the outside background flow is also of Taylor-Couette type, i.e.
	\begin{equation*}
		w_- (s) = a e^{-2s} + b,
	\end{equation*}
	where $a, b \in \R$ are both constants. It is clear that $\zeta_-$ is given by the solution formula
	\begin{equation*}
		\zeta_-(s) = \frac{1}{1-\Rout^{2\ak}}\exp(\abs{k}s) - \frac{\Rout^{2\ak}}{1-\Rout^{2\ak}}\exp(-\ak s) \qfor 0 \le s \le \log\Rout.
	\end{equation*}
	Here we note that $\Rout > 1$. Assume further that the interior region is an annulus and the background flow is also the Taylor-Couette flow, i.e., $0 < \Rin < 1$ and
	\begin{equation*}
		w_+(s) = Ae^{-2s} + B.
	\end{equation*}
	Then, the dispersive relation \eqref{dispersion} can be rewritten as
	\begin{equation}\label{disp TC}
		\begin{split}
			&\ak\qty(\frac{1+\Rin^{2\ak}}{1-\Rin^{2\ak}}+\varepsilon\cdot\frac{1+\Rout^{2\ak}}{\Rout^{2\ak}-1})\qty[c - \frac{\ak\cdot\frac{1+\Rin^{2\ak}}{1-\Rin^{2\ak}}(A+B) - B +\varepsilon\ak\cdot\frac{1+\Rout^{2\ak}}{\Rout^{2\ak}-1}(a+b) + \varepsilon b}{\ak\qty(\frac{1+\Rin^{2\ak}}{1-\Rin^{2\ak}}+\varepsilon\cdot\frac{1+\Rout^{2\ak}}{\Rout^{2\ak}-1})}]^2 \\
			&\quad = \frac{\alpha}{\varrho_+}\qty(k^2-1) + \frac{\qty[\ak\cdot\frac{1+\Rin^{2\ak}}{1-\Rin^{2\ak}}(A+B) - B +\varepsilon\ak\cdot\frac{1+\Rout^{2\ak}}{\Rout^{2\ak}-1}(a+b) + \varepsilon b]^2}{\ak\qty(\frac{1+\Rin^{2\ak}}{1-\Rin^{2\ak}}+\varepsilon\cdot\frac{1+\Rout^{2\ak}}{\Rout^{2\ak}-1})} + \qty(B^2-A^2) + \\
			&\qquad - \ak\frac{1+\Rin^{2\ak}}{1-\Rin^{2\ak}}(A+B)^2 + \varepsilon\qty(a^2-b^2) - \ak\varepsilon\frac{1+\Rout^{2\ak}}{\Rout^{2\ak}-1}(a+b)^2.
		\end{split}
	\end{equation}
	The sign of the right hand side expression will determine the stability of $k$-waves. Heuristically, once those physical parameters are fixed, the right hand side of \eqref{disp TC} can be viewed as
	\begin{equation*}
		\frac{\alpha}{\varrho_+} k^2 + \order{\ak},
	\end{equation*}
	which indicates the stabilization effect of capillary forces, as all waves with large wave numbers are stable.
	
	On the other hand, when taking $A = a = 0$ and $B = b$, one would obtain
	\begin{equation*}
		\boxed{\text{R.H.S. of } \eqref{disp TC}} = \frac{\alpha}{\varrho_+}\qty(k^2 -1) - b^2(1-\varepsilon)\qty[1-\frac{1-\varepsilon}{\ak(1+\varepsilon)}].
	\end{equation*}
	If $\varepsilon < 1$ and $b \gg 1$, the surface wave is linearly unstable for small wave numbers, although the velocities and vorticities are both continuous across the interface. Comparing this with Example \ref{ex TC flow} reveals the distinct dynamics of two-phase and one-phase flows.
	\end{ex}
	
	\section{Instabilities and Critical Layers}
	
	\subsection{Locations of Unstable Modes}\label{sec semi-circle}
	We first establish a semi-circle type result (see also \cite{Howard1961} and \cite{Liu2024}) on the location of possible unstable modes in terms of the range of angular velocities of background flows. From now on, we always assume that $\varrho_+ \ge \varrho_-$.
	
	\begin{theorem}
		Assume that $(c, k)$ is a pair of constants with $c = \cR + i\cI \in \C \setminus \R$ and $k \in \Z\setminus\{0\}$, and the ODE boundary value problems \eqref{BP} with dispersive relation \eqref{dispersion} admit non-trivial solutions. Denote by
		\begin{equation*}
			m \coloneqq \inf \qty\big(\qty{w_-(s) \colon 0 \le s < \log\Rout} \cup \qty{w_+(s) \colon \log\Rin < s \le 0})
		\end{equation*}
		and
		\begin{equation*}
			M \coloneqq \sup \qty\big(\qty{w_-(s) \colon 0 \le s < \log\Rout} \cup \qty{w_+(s) \colon \log\Rin < s \le 0}).
		\end{equation*}
		Then, for all $\ak \ge 1$ and $\varrho_+ \ge \varrho_-$, there holds
		\begin{equation*}
			\alpha\qty(k^2 -1) > mM(\varrho_+ - \varrho_-)
			\implies
			\qty(\cR - \frac{m+M}{2})^2 + \cI^2 < \qty(\frac{M-m}{2})^2.
		\end{equation*}
		Moreover, if $\ak \ge 2$ or $\varrho_+ > \varrho_-$, it holds that
		\begin{equation*}
			\alpha\qty(k^2 -1) \ge mM(\varrho_+ - \varrho_-)
			\implies
			\qty(\cR - \frac{m+M}{2})^2 + \cI^2 \le \qty(\frac{M-m}{2})^2.
		\end{equation*}
	\end{theorem}
	
	\begin{proof}
		Since $c \in \C\setminus \R$, it is legitimate to define the complex valued functions:
		\begin{equation}
			\chi_\pm(s) \coloneqq \frac{w_\pm(0)-c}{w_\pm(s) - c} \cdot \zeta_\pm (s).
		\end{equation}
		Then, $\chi_\pm$ satisfies the ODE:
		\begin{equation}\label{eq chi}
			- \dv{s}\qty[(w_\pm(s)-c)^2\dot{\chi}_\pm] + k^2\qty(w_\pm(s) - c)^2 \chi_\pm + \dv{s}\qty[(w_\pm(s)-c)^2]\chi_\pm = 0,
		\end{equation}
		with boundary conditions:
		\begin{equation}
			\chi_\pm(0) = 1 \qand \chi_+(\log\Rin) = \chi_-(\log\Rout) = 0.
		\end{equation}
		The dispersive relation \eqref{dispersion} can be written as
		\begin{equation}\label{disp chi}
			\begin{split}
				&\varrho_+ \dot{\chi}_+(0)\qty[w_+(0)-c]^2 - \varrho_- \dot{\chi}_-(0)\qty[w_-(0)-c]^2 \\
				&\quad = \alpha\qty(k^2 -1) + \varrho_+w_+^2(0)-\varrho_-w_-^2(0) + 2\qty[\varrho_-w_-(0) - \varrho_+w_+(0)]c.
			\end{split}
		\end{equation}
		
		Multiplying \eqref{eq chi} by $\chi_\pm^*$ (here ${}^*$ represents the complex conjugate) and integrating over the interval $(\log\Rin, 0)$ or $(0, \log\Rout)$ leads to
		\begin{equation}\label{int chi-}
			-\qty[w_-(0)-c]^2\dot{\chi}_-(0) + \qty[w_-(0)-c]^2 = \int_0^{\log\Rout} \qty[w_-(s)-c]^2 \underbrace{\qty(k^2\abs{\chi_-}^2 + \abs{\dot{\chi}_-}^2 - 2\ev{\chi_-, \dot{\chi}_-})}_{\eqqcolon X_-} \dd{s}
		\end{equation}
		and
		\begin{equation}\label{int chi+}
			\qty[w_+(0)-c]^2 \dot{\chi}_+(0) - \qty[w_+(0)-c]^2 = \int_{\log\Rin}^0 \qty[w_+(s)-c]^2 \underbrace{\qty(k^2\abs{\chi_+}^2 + \abs{\dot{\chi}_+}^2 - 2\ev{\chi_+, \dot{\chi}_+})}_{\eqqcolon X_+} \dd{s},
		\end{equation}
		where $\ev{\cdot, \cdot}$ represents the inner product in $\R^2 \simeq \C$. More precisely, for $z_j = x_j + i y_j \in \C $, $ (j =1, 2)$ with $x_j, y_j \in \R$, we define
		\begin{equation*}
			\ev{z_1, z_2} \coloneqq x_1x_2 + y_1 y_2 = \frac{1}{2}\qty(z_1z_2^* + z_1^*z_2).
		\end{equation*}
		Then, it is clear that $X_\pm \ge 0$, and when $\ak \ge 2$, $X_\pm = 0$ iff $\chi_\pm \equiv 0$.
		
		Combining \eqref{disp chi} with \eqref{int chi-}-\eqref{int chi+}, it is routine calculate that
		\begin{equation}
			\begin{split}
				&\int_{\log\Rin}^0 \varrho_+\qty[w_+(s)-c]^2 X_+(s) \dd{s} + \int_{0}^{\log\Rout} \varrho_- \qty[w_-(s)-c]^2 X_-(s) \dd{s} \\
				&\quad = \alpha\qty(k^2 -1) + (\varrho_- - \varrho_+)c^2
			\end{split}
		\end{equation}
		Taking the imaginary and real parts of both sides, one obtains that
		\begin{equation}\label{im chi}
			\begin{split}
				&\int_{\log\Rin}^0 \varrho_+\qty[\cR-w_+(s)]X_+(s) \dd{s} + \int_0^{\log\Rout}\varrho_-\qty[\cR - w_-(s)] X_-(s) \dd{s}\\
				&\quad = \cR(\varrho_- - \varrho_+)
			\end{split}
		\end{equation}
		and
		\begin{equation}\label{re chi}
			\begin{split}
				&\int_{\log\Rin}^0 \varrho_+\qty(\qty[w_+(s)-\cR]^2 - \cI^2)X_+(s) \dd{s} + \int_0^{\log\Rout}\varrho_- \qty(\qty[w_-(s)-\cR]^2 - \cI^2) X_-(s) \dd{s} \\
				&\quad = \alpha\qty(k^2 -1) + (\varrho_- - \varrho_+)\qty(\cR^2 - \cI^2).
			\end{split}
		\end{equation}
		
		On the other hand, it is clear that
		\begin{equation}\label{rel w m M}
			\qty\big[w_\pm(s) - m] \cdot \qty\big[w_\pm(s) - M] \le 0 \quad \forall\, s.
		\end{equation}
		Then, the elementary relation
		\begin{equation*}
			\begin{split}
				&(w-m)(w-M) \\
				&\quad = \qty[(w-\cR)^2 - \cI^2] + (m+M-2\cR)(\cR-w) + \qty(\cR-\frac{m+M}{2})^2 + \cI^2 - \qty(\frac{M-m}{2})^2,
			\end{split}
		\end{equation*}
		together with \eqref{im chi}-\eqref{rel w m M}, yield that
		\begin{equation}
			\begin{split}
				0 &\ge \int_{\log\Rin}^0 \varrho_+ \qty[w_+(s)-m]\qty[w_+(s)-M]X_+(s) \dd{s} + \int_0^{\log\Rout} \varrho_- \qty[w_-(s)-m]\qty[w_-(s)-M]X_-(s) \dd{s} \\
				&= \alpha\qty(k^2-1) + (\varrho_- - \varrho_+)\qty(\cR^2 - \cI^2) + \qty(m+M-2\cR)\cR(\varrho_- - \varrho_+) + \\
				&\qquad + \qty[\qty(\cR-\frac{m+M}{2})^2 + \cI^2 - \qty(\frac{M-m}{2})^2] \qty(\int_{\log\Rin}^0 \varrho_+ X_+(s) \dd{s} + \int_{0}^{\log\Rout} \varrho_- X_-(s) \dd{s}) \\
				&= \qty[\qty(\cR-\frac{m+M}{2})^2 + \cI^2 - \qty(\frac{M-m}{2})^2] \qty((\varrho_+ - \varrho_-) + \int_{\log\Rin}^0 \varrho_+ X_+(s) \dd{s} + \int_{0}^{\log\Rout} \varrho_- X_-(s) \dd{s}) + \\
				&\qquad + \alpha\qty(k^2 - 1) + mM(\varrho_- - \varrho_+)
			\end{split}
		\end{equation}
		Note that $\varrho_+ \ge \varrho_-$ and $X_\pm \ge 0$, this concludes the proof.
	\end{proof}
	
	\subsection{Capillary Water Waves near Taylor-Couette Flows}\label{sec TC flow}
	From now on, we assume that $\varrho_+ > 0$ and the flow in the inner annulus is of the  Taylor-Couette type, i.e.,
	\begin{equation}
		\vV_+ = \qty(\frac{A}{r} + Br) \etheta,
	\end{equation}
	for constants $A, B \in \R$ ($A = 0$ if $\cU^+_*$ is the unit disk, for which $\Rin = 0$). We denote by
	\begin{equation}
		\varepsilon \coloneqq \frac{\varrho_-}{\varrho_+}
	\end{equation}
	the density ratio.	Then, the dispersive relation \eqref{dispersion} can be written as
	\begin{equation}\label{disp delta}
		\begin{split}
			\frac{\alpha}{\varrho_+}\qty(k^2 -1) = &-\varepsilon\dot{\zeta}_-(0)\qty\big[c-w_-(0)]^2 -\varepsilon\qty\big[\varpi_-(0)c - w_-(0)\varpi_-(0) + w_-(0)w_-(0)]+ \\
			&+\frac{\ak\qty(1+\Rin^{2\ak})}{1-\Rin^{2\ak}} \qty[c-\qty(A+B-\frac{\qty(1-\Rin^{2\ak})B}{\ak\qty(1+\Rin^{2\ak})})]^2- \frac{\qty(1-\Rin^{2\ak})B^2}{\ak\qty(1+\Rin^{2\ak})}+ (A+B)^2,
		\end{split}
	\end{equation}
	which can be viewed as an algebraic equation for $c$ with multiple parameters.
	It is clear that, when $\varepsilon = 0$ and
	\begin{equation*}
		\frac{\alpha}{\varrho_+}\qty(k^2 -1) + \frac{\qty(1-\Rin^{2\ak})B^2}{\ak\qty(1+\Rin^{2\ak})} > (A+B)^2,
	\end{equation*}
	there exist two distinct real roots of the algebraic equation \eqref{disp delta}:
	\begin{equation}\label{def ck}
		c_\pm^{(k)} = \qty(A+B-\frac{1-\Rin^{2\ak}}{\ak(1+\Rin^{2\ak})}\cdot B) \pm \sqrt{\frac{1-\Rin^{2\ak}}{\ak(1+\Rin^{2\ak})}\qty[\frac{\alpha}{\varrho_+}\qty(k^2-1) + \frac{1-\Rin^{2\ak}}{\ak(1+\Rin^{2\ak})}\cdot B^2 - (A+B)^2]}.
	\end{equation}
	When the background flow $\vV_-$, the density $\varrho_+$, and the surface tension coefficient $\alpha$ are fixed, the algebraic equation \eqref{disp delta} (for $c$) has only two parameters, say, $\varepsilon$ and $\varepsilon\dot{\zeta}_-(0)$. Therefore, the two solutions to \eqref{disp delta} can be expressed analytically in terms of them. More precisely, one has
	\begin{equation}\label{eqn c}
		c_\pm = h_{\text{R}}^\pm\qty\Big(\varepsilon\Re{\dot{\zeta}_-(0)}, \varepsilon\Im{\dot{\zeta}_-(0)}, \varepsilon) + i\varepsilon\Im{\dot{\zeta}_-(0)}\cdot h_{\text{I}}^\pm\qty\Big(\varepsilon\Re{\dot{\zeta}_-(0)}, \varepsilon\Im{\dot{\zeta}_-(0)}, \varepsilon),
	\end{equation}
	where $h_{\text{R}}^\pm$ and $h_{\text{I}}^\pm$ are all real-valued analytic functions satisfying at $(0, 0, 0)$ (following from the implicit function theorem):
	\begin{equation}\label{h0}
		h_{\text{R}}^\pm(0) = c^{(k)}_{\pm} \qand h_{\text{I}}^\pm (0) = \frac{\qty[c^{(k)}_\pm - w_-(0)]^2}{\frac{2\ak\qty(1+\Rin^{2\ak})}{1-\Rin^{2\ak}} \qty[c^{(k)}_\pm-\qty(A+B-\frac{\qty(1-\Rin^{2\ak})B}{\ak\qty(1+\Rin^{2\ak})})]}.
	\end{equation}
	Here $c^{(k)}_\pm$ are the two real roots given by \eqref{def ck}. In particular, one has
	\begin{equation}
		h_{\text{I}}^+(0) > 0 \qand h_{\text{I}}^-(0) < 0,
	\end{equation}
	whenever $c^{(k)}_\pm \neq w_-(0)$.
	
	In the sequel, we shall study the instability induced by the outer regions with $\varepsilon \ll 1$, which characterizes the dynamics of air-water surface waves. The background water flow is assumed to be a Taylor-Couette one. We shall search for necessary  conditions and sufficient ones for circular wind profiles that would induce unstable surface waves. Then, the ODE boundary value problems \eqref{BP} can be rewritten as:
	\begin{equation}\label{eqn zeta}
		\begin{cases*}\displaystyle
			-\ddot{\zeta}_-(s) + \qty(k^2 + \frac{2\dot{w}_-(s) + \ddot{w}_-(s)}{w_-(s) - c})\zeta_-(s) = 0 \qfor 0 < s < \log\Rout, \\
			\zeta_-(0) = 1 \qc \zeta_-(\log\Rout) = 0,
		\end{cases*}
	\end{equation}
	and the dispersive relation is given by \eqref{eqn c}.

	\subsection{Necessity of Critical Layers for Unstable Waves}
	Suppose that $\varepsilon \ll 1$, then for each fixed wave number $k$ and wind profile $w_-(s)$, the phase velocity $c_\pm$ given by the dispersive relation \eqref{disp delta} (and hence formula \eqref{eqn c}) should be close to $c_\pm^{(k)}$ in \eqref{def ck}. Thus, the relation between $c^{(k)}_\pm$ and $w_-$ would be crucial for the (in-)stability of surface waves. Indeed, whether $c^{(k)}_\pm$ belong to the range of $w_-$ would evidently influence the linear stability, which is more precisely demonstrated in the following proposition:
	\begin{prop}
		Suppose that $w_- \in C^2$ and $c^{(k)}_+ \notin w_-([0, \log\Rout])$. Then, there exists a constant $\varepsilon_k > 0$, so that if the boundary value problem \eqref{eqn zeta} with \eqref{eqn c} is solvable for $0 < \varepsilon < \varepsilon_k$ and mode $(k, c)$ satisfying $\ak \ge 2$ and
		\begin{equation*}
			\abs{c-c^{(k)}_+} \le \frac{1}{2} \min_{0 \le s \le \log\Rout} \abs{c^{(k)}_+ - w_-(s)},
		\end{equation*}
		then $c \in \R$. The same result also holds when $c^{(k)}_+$ replaced by $c^{(k)}_-$.
	\end{prop}
	
	\begin{proof}
		Since $c^{(k)}_+$ does not belong to the range of $w_-(s)$, one can define the function
		\begin{equation}
			\eta(s) \coloneqq \frac{w_-(0)-c}{w_-(s) -c} \qty[\zeta_-(s) + \frac{s-\log\Rout}{\log\Rout}].
		\end{equation}
		Then, $\eta(s)$ satisfies the ODE
		\begin{equation}\label{eqn eta}
				\begin{split}
					&-\dv{s}\qty[(w_-(s) - c)^2 \dot{\eta}] + k^2\qty[w_-(s) - c]^2\eta + \dv{s}\qty[(w_-(s) - c)^2] \eta \\
					&\quad = \qty[k^2(w_-(s) - c) + 2\dot{w}_-(s) + \ddot{w}_-(s)]\qty[w_-(0)-c]\cdot\frac{s - \log\Rout}{\log\Rout} \\
					&\quad\eqqcolon \gamma(s),
				\end{split}
		\end{equation}
		together with the boundary conditions
		\begin{equation}
			\eta(0) = \eta(\log\Rout) = 0.
		\end{equation}
		Multiplying \eqref{eqn eta} by $\eta^*$ and integrating by parts yield that
		\begin{equation}\label{est eta 0}
			\int_{0}^{\log\Rout} \qty\big(w_-(s)-c)^2 \qty(\abs{\dot{\eta}}^2 + k^2\abs{\eta}^2 - 2\ev{\eta, \dot{\eta}}) \dd{s} = \int_{0}^{\log\Rout} \gamma(s)\eta^*(s) \dd{s}.
		\end{equation}
		Denote by
		\begin{equation*}
			\ell \coloneqq \min_{s \in [0, \log\Rout]} \abs{c^{(k)}_+ - w_-(s)}.
		\end{equation*}
		It is clear that
		\begin{equation*}
			\abs{\qty\big(w_-(s)-c)^2} \gtrsim \ell^2 \qq{for all} 0 \le s \le \log\Rout.
		\end{equation*}
		Then, through taking real or imaginary parts of \eqref{est eta 0}, it follows that
		\begin{equation}
			\ell^2\qty(k^2\norm*{\eta}_{L^2}^2 + \norm*{\dot\eta}_{L^2}^2) \lesssim \int_0^{\log\Rout} \abs*{\gamma}(s) \cdot \abs*{\eta}(s) \dd{s}.
		\end{equation}
		Note that $\gamma(s)$ is completely determined by $w_-$ and $\Rout$, it follows from the Cauchy-Schwartz inequality that
		\begin{equation*}
			\norm*{\eta}_{L^2} + k^{-1}\norm*{\dot\eta}_{L^2} \lesssim \ell^{-2},
		\end{equation*}
		where the implicit constant depends on $\Rout$ and $\norm{w_-}_{C^2}$. The construction of $\eta$ yields the estimate
		\begin{equation*}
			\norm*{\zeta_-}_{L^2} + k^{-1}\norm*{\dot{\zeta}_-}_{L^2} \lesssim \ell^{-2},
		\end{equation*}
		
		On the other hand, note that the function $\Phi$ defined through
		\begin{equation}\label{def Phi}
			\Phi(s) \coloneqq \frac{i}{2}\qty(\zeta_- \dot{\zeta}_-^* - \dot{\zeta}_-\zeta_-^*)
		\end{equation}
		satisfies
		\begin{equation}\label{ODE Phi}
			\begin{cases*}
				\displaystyle
				\dv{s}\Phi = \frac{\abs{\zeta_-}^2(2\dot{w}_- + \ddot{w}_-)}{\abs{w_- - c}^2}\cdot\cI, \\
				\Phi(0) = \Im{\dot{\zeta}_-(0)} \qc \Phi(\log\Rout) = 0.
			\end{cases*}
		\end{equation}
		Particularly, fundamental theorem of Calculus implies that
		\begin{equation*}
			\abs{\Im{\dot{\zeta}_-(0)}} \lesssim \ell^{-6}\abs{\cI},
		\end{equation*}
		which, together with the relation \eqref{eqn c}, yield
		\begin{equation*}
			\abs{\cI} \lesssim \varepsilon \ell^{-6} \abs{\cI}.
		\end{equation*}
		Thus, as long as $\varepsilon \ll 1$, one obtains $\cI = 0$.
	\end{proof}
	
	\begin{remark}
		The above proposition indicates that, for $C^2$-smooth wind profile $w_-$, the unstable mode can only bifurcate from the the range of $w_-$. If $w_-$ is not $C^2$, there might be unstable modes lurking elsewhere, which is indicated briefly in \S\ref{sec other instability} (see also \cite[\S5]{BuhlerSWZ2016}).
	\end{remark}

	\subsection{Instability Induced by Critical Layers}\label{sec CL}
	
	Now, we turn to show that critical layers can indeed lead to unstable modes for sufficiently smooth wind profiles. Regarding \eqref{ODE Phi}, one may first formally derive that
	\begin{equation*}
		\Im{\dot{\zeta}_-(0)} = - \int_0^{\log\Rout} \cI \cdot \frac{\dot{\varpi}_-(s)\cdot\abs{\zeta_-(s)}^2}{\abs{w_-(s) - c}^2} \dd{s} = -\cI \int_0^{\log\Rout}  \frac{\dot{\varpi}_-(s)\cdot\abs{\zeta_-(s)}^2}{\qty[w_-(s) - \cR]^2 + \cI^2} \dd{s}.
	\end{equation*}
	Assume that $\cR$ is a regular value of $w_-$, and $\abs{\zeta_-}^2, \abs{\dot{w}_-}, \dot{\varpi}_-$ are all slowly varying.
	First note that, for a family of functions parameterized by $\nu > 0$:
	\begin{subequations}\label{conver delta}
		\begin{equation}
			f_\nu(t) \coloneqq \frac{\nu}{1+(\nu t)^2} \I_{[-t_0, t_0]}(t),
		\end{equation}
		where $t_0$ is an arbitrary positive constant, there holds
		\begin{equation}
			f_\nu \to \pi\vdelta_0 \qas \nu \to 0 \qin \mathcal{D}'(\R),
		\end{equation}
	\end{subequations}
	for which $\vdelta_0$ is the Dirac mass. 
	Thus, one may compute that, near a regular point $\sigma$ for which $w_-(\sigma)=\cR$ and $\dot{w}_-(\sigma) \neq 0$, there holds
	\begin{equation*}
		\begin{split}
			&-\cI\int_{\sigma-\epsilon_0}^{\sigma+\epsilon_0} \frac{\dot{\varpi}_-(s)\cdot\abs{\zeta_-(s)}^2}{\qty[w_-(s) - \cR]^2 + \cI^2} \dd{s} \\
			& \quad \approx -\cI\int_{\sigma-\epsilon_0}^{\sigma+\epsilon_0} \frac{\dot{\varpi}_-(\sigma)\cdot\abs{\zeta_-(\sigma)}^2}{\abs{\dot{w}_-(\sigma)}^2\abs{s-\sigma}^2 + \cI^2} \dd{s} \\
			& \quad \approx - \frac{\dot{\varpi}_-(\sigma) \cdot \abs{\zeta_-(\sigma)^2}}{\abs{\dot{w_-}(\sigma)}} \int_{\sigma-\epsilon_0}^{\sigma+\epsilon_0} \frac{\frac{\abs{w}_-(\sigma)}{\cI} \dd{s}}{1 + \qty[\frac{\abs{\dot{w}_-(\sigma)}}{\cI}(s-\sigma)]^2} \\
			&\quad \xrightarrow{\cI \to 0} - \sgn(\cI) \pi \frac{\dot{\varpi}_-(\sigma) \cdot \abs{\zeta_-(\sigma)^2}}{\abs{\dot{w_-}(\sigma)}}.
		\end{split}
	\end{equation*}
	Particularly, one would obtain that
	\begin{equation*}
		\Im{\dot{\zeta}_-(0)} \asymp - \sgn(\cI) \pi \sum_{\sigma \in (w_-)^{-1}(\{\cR\})}\frac{\dot{\varpi}_-(\sigma) \cdot \abs{\zeta_-(\sigma)}^2}{\abs{\dot{w_-}(\sigma)}} \qq{as} \cI \asymp 0.
	\end{equation*}
	
	On the other hand, when $\varepsilon \ll 1$, i.e., the density of air is sufficiently small when compared to that of water, the phase velocities of $k$-waves ought to be close to the two roots given in \eqref{def ck}, which are the corresponding eigenvalues for the water-vacuum problems. In view of the previous heuristic arguments and the dispersive relation \eqref{eqn c}, one may infer that the unstable mode (if exists) would bifurcate from $c^{(k)}_\pm$, as long as the evaluations of $\dot{\varpi}_-$ in the preimage of $c^{(k)}_\pm$ admit the same sign. More precisely, there holds the following result:
	\begin{theorem}\label{thm instability}
		Suppose that $w_- \in C^4$, $1<\Rout<\infty$, and $c^{(k)}_+ \in \R$ given by \eqref{def ck} is a regular value of $w_-$, i.e.,
		\begin{equation*}
			\Set*{s \given w_-(s) = c^{(k)}_+} \eqqcolon \{s_1, \dots, s_n\} \subset (0, \log\Rout), \qq{with} \dot{w}_-(s_j) \neq 0 \; (1 \le j \le n).
		\end{equation*}
		Assume further that
		\begin{equation*}
			c^{(k)}_+ \dot{\varpi}_-(s_j) \le 0 \qfor 1\le j \le n, \qand c^{(k)}_+\dot{\varpi}_-(s_l) < 0 \qfor l = n-1 \text{ or } n.
		\end{equation*}
		Then, for $\varepsilon \coloneqq {\varrho_-}/{\varrho_+} \ll 1$, there exists a phase velocity $c \in \C\setminus\R$ satisfying $\abs{c-c^{(k)}_+} = \order{\varepsilon}$ and $\Im{c} > 0$ with a positive lower bound, so that the ODE problem \eqref{eqn zeta} together with the dispersive relation \eqref{eqn c} is solvable. Namely, the wind-perturbed water waves are linearly unstable. The same results hold with all $c^{(k)}_+$ replaced by $c^{(k)}_-$.
	\end{theorem}
	
	Before preceding to the proof, we first give some preparations on the solvability of ODE boundary value problems near singularities of the coefficients. Regarding \eqref{def Phi}-\eqref{ODE Phi} and \eqref{eqn c}, one may consider the quantities (see also \cite{BuhlerSWZ2016}):
	\begin{equation}\label{def xi}
		\xi_1 \coloneqq \abs{\zeta_-}^2\qc \xi_2 \coloneqq \frac{1}{2}\qty(\dot{\zeta}_-\zeta_-^* +  \dot{\zeta}_-^*\zeta_-) \qc \xi_3 \coloneqq \abs*{\dot{\zeta}_-}^2 \qc \Phi \coloneqq \frac{i}{2}\qty(\zeta_- \dot{\zeta}_-^* - \dot{\zeta}_-\zeta_-^*).
	\end{equation}
	It is clear that these four quantities are all real-valued functions. For the simplicity of notations, we denote by $\vXi \coloneqq (\xi_1, \xi_2, \xi_3)^T$. Then, it follows from \eqref{BP} that (where $\varpi_- \equiv 2w_- + \dot{w}_-$)
	\begin{equation}\label{ode xi}
		\dv{s} \bmqty{\xi_1 \\ \xi_2 \\ \xi_3 \\ \Phi} = \bmqty{2 \xi_2 \\ \qty(k^2 + \frac{\dot{\varpi}_- (w_- - \cR)}{(w_- - \cR)^2 + \cI^2})\xi_1 + \xi_3 \\ 2\qty(k^2 + \frac{\dot{\varpi}_- (w_- - \cR)}{(w_- - \cR)^2 + \cI^2})\xi_2 + \frac{2\cI\dot{\varpi}_-}{(w_- - \cR)^2 + \cI^2} \Phi \\ \frac{\cI \dot{\varpi}_-}{(w_- - \cR)^2 + \cI^2} \xi_1}.
	\end{equation}
	Moreover, due to the construction \eqref{def xi}, there holds the identity
	\begin{equation}
		(\xi_2)^2 + \Phi^2 - \xi_1\xi_3 = 0,
	\end{equation}
	which is also preserved by the evolution equation \eqref{ode xi}.
	
	Since we are considering the (in-)stability of wind-perturbed water waves, we would like to study the behavior of \eqref{ode xi} with small $\cI$ and its limiting process. When $\cI \asymp 0$, the ODE system \eqref{ode xi} would become
	\begin{equation}\label{limit ode}
		\dv{s} \bmqty{\wh{\xi}_1 \\ \wh{\xi}_2 \\ \wh{\xi}_3 \\ \wh{\Phi}} = \bmqty{2 \wh{\xi}_2 \\ \qty(k^2 + \frac{\dot{\varpi}_-}{w_- - \cR})\wh{\xi}_1 + \wh{\xi}_3 \\ 2\qty(k^2 + \frac{\dot{\varpi}_- }{(w_- - \cR)})\wh{\xi}_2 + \frac{2\cI\dot{\varpi}_-}{(w_- - \cR)^2} \wh{\Phi} \\ 0}.
	\end{equation}
	However, the coefficient matrix of \eqref{limit ode} has singularities when $\cR$ belongs to the range of $w_-$. Fortunately, when $c$ is sufficiently close to a regular value of $w_-$, the limiting system \eqref{limit ode} is still solvable in a neighborhood of the corresponding regular point. 
	
	More precisely, assume that $w_- \in C^3$ and $s_0 \in (0, \log\Rout)$ is a regular point of $w_-$, i.e. $\dot{w}_-(s_0) \neq 0$. Then, there is a constant $\delta > 0$, so that
	\begin{equation}\label{delta}
		\frac{1}{2} < \frac{\dot{w}_-(s)}{\dot{w}_-(s_0)} < 2 \qq{for} s_0 - \delta < s < s_0 + \delta.
	\end{equation}
	Suppose that the phase velocity $c$ satisfies $\cR = w_-(s')$ for some $s'$, and there hold
	\begin{equation}\label{s'}
		\abs{s' - s_0} \ll \delta \qc \abs{\cI} \ll \delta.
	\end{equation}
	As the coefficient matrix of \eqref{limit ode} is singular (only) at $s = s'$ for $s \in (s_0-\delta, s_0 + \delta)$, one needs to impose jump conditions on the solutions at $s'$. First observe that, due to the assumptions \eqref{delta}-\eqref{s'}, one may write
	\begin{equation}
		\frac{\dot{\varpi}_-(s)\cI}{[w_-(s)-w_-(s')]^2 + \cI^2} \approx \frac{\dot{\varpi}_-(s')\cI}{\dot{w}_-(s')^2(s-s')^2 + \cI^2} \approx \frac{\dot{\varpi}_-(s')}{\abs{\dot{w}_-(s')}} \cdot \frac{\frac{\abs{\dot{w}_-(s')}}{\cI}}{1 + \qty[\frac{\dot{w}_-(s')}{\cI}(s-s')]^2}.
	\end{equation}
	Particularly, thanks to \eqref{conver delta}, it is natural to impose the jump condition on $\wh{\Phi}$ as:
	\begin{subequations}\label{jump condition}
		\begin{equation}
			\wh{\Phi}(s'+0) - \wh{\Phi}(s' -0) = \sgn(\cI) \cdot \frac{\pi\dot{\varpi}_-(s')\wh{\xi}_1(s')}{\abs{w_-(s')}}.
		\end{equation}
		Similarly, the jump conditions on $\wh{\vXi}$ is supposed to be:
		\begin{equation}
			\wh{\xi}_1(s'+0) = \wh{\xi}_1(s'-0) \qc \wh{\xi}_2(s'+0) = \wh{\xi}_2(s'-0),
		\end{equation}
		and
		\begin{equation}
			\wh{\xi}_3(s'+0)-\wh{\xi}_3(s'-0) = \sgn(\cI)\cdot\frac{\pi\dot{\varpi}_-(s')}{\abs{w_-(s')}}\qty(\wh{\Phi}(s'+0)+\wh{\Phi}(s'-0)).
		\end{equation}
	\end{subequations}
	Concerning the solvability of \eqref{limit ode} with jump conditions \eqref{jump condition}, there holds the following result (cf. \cite[Proposition 4.2]{BuhlerSWZ2016}):
	\begin{lemma}
		Suppose that $0 < \mu < 1$, $0 < \delta' < \delta$ are fixed constants, $0 < \abs{\cI} \ll \delta$, and $(\vXi, \Phi)$ is a solution to \eqref{ode xi} on $[s' - \delta', s' + \delta']$ satisfying the bound
		\begin{equation*}
			\abs{\xi_1(s' + \delta')}^2 + \abs{\xi_2(s' + \delta')}^2 + \abs{\xi_3(s' + \delta')}^2 + \abs{\Phi(s' + \delta')}^2 \le 1.
		\end{equation*}
		Then, the ODE system \eqref{limit ode} together with the jump condition \eqref{jump condition} admits a unique solution with boundary value
		\begin{equation*}
			\qty\big(\vXi, \Phi)(s' + \delta') = \qty(\wh{\vXi}, \wh{\Phi})(s' + \delta').
		\end{equation*}
		Moreover, the solution satisfies the estimates
		\begin{equation*}
			\abs{\xi_1(\sigma) - \wh{\xi}_1(\sigma)} \lesssim \abs{\cI} \cdot \abs{\log^3\abs{\cI}} \qfor \sigma\in[s'-\delta', s' + \delta'],
		\end{equation*}
		and
		\begin{equation*}
			\abs{\vXi(s' - \delta') - \wh{\vXi}(s' - \delta')} \lesssim \abs{\cI}^{\mu} \qc \abs{\Phi(s'-\delta') - \wh{\Phi}(s'-\delta')} \lesssim \abs{\cI}\cdot\abs{\log^3\abs{\cI}}.
		\end{equation*}
		Alternatively, one may assume the pointwise bound of $(\vXi, \Phi)$ at the left end point $s = s'-\delta'$. Then, there hold similar estimate on the right endpoint $s = s'+\delta'$.
	\end{lemma}
	
	Now, assume that $c_*$ is a regular value of $w_-$. Then, it follows that the preimage of $c_*$ under $w_-$ is a discrete set. Namely, there holds
	\begin{equation}\label{c*}
		(w_-)^{-1}(\{c_*\}) = \{\sigma_1, \dots, \sigma_n\} \subset (0, \log\Rout), \qq{where} \dot{w}_-(\sigma_j) \neq 0 \; (1 \le j \le n).
	\end{equation}
	We may assume that $0 < \sigma_1 < \cdots < \sigma_n < \log\Rout$. Concerning the solutions to \eqref{ode xi} with $c \in \C\setminus\R$ and $\abs{c-c_*} \ll 1$, it holds that (cf. \cite[Proposition 4.3]{BuhlerSWZ2016}):
	\begin{lemma}\label{lem 2}
		Let $0 < \mu < 1$ be a fixed parameter. Then, there exist constants $C, \varepsilon_0, \delta_0$ depending on $\mu, k, \norm{w_-}_{C^3}$, and $\max_{j} \abs{\dot{w}_-(\sigma_j)}^{-1}$, so that
		\begin{equation*}
			\sigma_1 - \delta_0 > 0 \qc \sigma_n = \delta_0 < \log\Rout \qc \sigma_{l}+\delta_0 < \sigma_{l+1} - \delta_0 \; (1 \le l \le n-1),
		\end{equation*}
		and for $c = \cR + i\cI$ with $\abs{c-c_*} + \abs{\cI}<\varepsilon_0$, the following estimates hold. Let $(\vXi, \Phi)$ be the solution to \eqref{ode xi} with boundary value
		\begin{equation*}
			\vXi(\log\Rout) = (0, 0, 1)^T \qc \Phi(\log\Rout) = 0,
		\end{equation*}
		and $\qty(\wh{\vXi}, \wh{\Phi})$ the solution to \eqref{limit ode} for $s \notin (w_-)^{-1}(\{\cR\})$ satisfying the boundary value
		\begin{equation*}
			\wh{\vXi}(\log\Rout) = (0, 0, 1)^T \qc \wh{\Phi}(\log\Rout) = 0,
		\end{equation*}
		and the jump conditions \eqref{jump condition} at all points in $(w_-)^{-1}(\{\cR\})$. Then, one has
		\begin{equation*}
			\abs{\xi_1(\sigma) - \wh{\xi}_1(\sigma)} \le C\abs{\cI}^{\mu} \qfor 0 \le \sigma \le \log\Rout,
		\end{equation*}
		and
		\begin{equation*}
			\abs{\vXi - \wh{\vXi}} + \abs{\Phi - \wh{\Phi}} \le C\abs{\cI}^\mu \qq{on} [0, \log\Rout] \setminus \bigcup_{1 \le j \le n}(\sigma_j - \delta_0, \sigma_j + \delta_0).
		\end{equation*}
	\end{lemma}
	In other words, for arbitrarily small $\cI$, solutions to \eqref{ode xi} are well-behaved away from the singular regions. Notice further that, when $c = \cR + i\cI$ is fixed, the singular issues of \eqref{limit ode} actually emerge near the points
	\begin{equation}
		(w_-)^{-1}(\{\cR\}) \eqqcolon \{\sigma_1' < \cdots < \sigma_n'\}.
	\end{equation}
	Indeed, if $w_-$ is sufficiently smooth, there holds the following refined result (cf. \cite[\S4.4]{BuhlerSWZ2016}):
	
	
	\begin{lemma}\label{lem 4}
		Assume that $w_- \in C^4$, $\dot{\varpi}_-(\sigma_j) \equiv [2\dot{w}_-(\sigma_j) + \ddot{w}_-(\sigma_j)]$ are all non-positive (or non-negative) and $\dot{\varpi}_-(\sigma_{l}) \neq 0$ for $l = n$ or $(n-1)$, here $\sigma_j$ are the points defined in \eqref{c*}. Then, for all $\cR$ close to $c_*$, the system \eqref{limit ode} admits a unique solution $(\wt{\vXi}, \wt{\Phi})$ satisfying the jump condition \eqref{jump condition} at all $\sigma_j'$, the boundary condition
		\begin{equation*}
			\wt{\xi}_1(0) = 1 \qc \wt{\xi}_1 (\log\Rout) = \wt{\xi}_2(\log\Rout) = \wt{\Phi}(\log\Rout) = 0,
		\end{equation*}
		and the relations $\wt{\xi}_1(\sigma_j') \neq 0 \; (1 \le j \le n)$,
		\begin{equation*}
			\wt{\Phi}(0) = - \sgn(\cI)\pi\sum_{1\le j \le n} \frac{\dot{\varpi}_-(\sigma_j')\wt{\xi}_1(\sigma_j')}{\abs*{\dot{w}_-(\sigma_j')}} \neq 0.
		\end{equation*}
		Moreover, when $\cR$ is viewed as a parameter, $(\wt{\vXi}, \wt{\Phi})_{\restriction s = 0}$ is $C^1$ with respect to $\cR$.
		
		On the other hand, if $c$ is sufficiently close to $c_*$ and $\abs{\cI}>0$, the boundary value problem \eqref{eqn zeta} admits a unique solution, which corresponds to a unique solution to \eqref{ode xi} satisfying
		\begin{equation*}
			\xi_1(0)=1 \qc \xi_1(\log\Rout)=\xi_2(\log\Rout)=\Phi(\log\Rout)=0 \qc \xi_3(\log\Rout)>0,
		\end{equation*}
		\begin{equation*}
			\dot{\zeta}_-(0) = \xi_2(0) + i \Phi(0),
		\end{equation*}
		and the error estimate (here $0 < \mu < 1$ is an arbitrary but fixed parameter, and the implicit constant depends on $\mu$):
		\begin{equation*}
			\abs{\vXi-\wt{\vXi}} + \abs{\Phi - \wt{\Phi}} \lesssim \abs{\cI}^{\mu}.
		\end{equation*}
	\end{lemma}
	
	Now, with the help of preceding lemmas, we turn to the proof of Theorem \ref{thm instability}.
	
	\begin{proof}[Proof of Theorem \ref{thm instability}]
		Let $(\wh{\vXi}^{\sharp}, \wh{\Phi}^{\sharp}) $ be the solution to \eqref{limit ode} with parameter $\cR \coloneqq c^{(k)}_+$, where $c^{(k)}_+$ is given by \eqref{def ck}. Denote by
		\begin{equation}
			c^{\sharp} \coloneqq - \pi h_{\text{I}}^+(0) \sum_{1\le j \le n} \frac{\dot{\varpi}_-(s_j) \wh{\xi}^\sharp_1(s_j)}{\abs{\dot{w_-}(s_j)}},
		\end{equation}
		where $h_\text{I}^+$ is the function in \eqref{eqn c} and $h_{\text{I}}^+(0)$ is given by \eqref{h0}. It follows from Lemma \ref{lem 4} that $c^\sharp > 0$. Define two complex-valued functions by:
		\begin{equation}
			\begin{split}
				\Lambda_1 (\nu_1, \varepsilon) &\coloneqq c^{(k)}_+ + \nu_1 - h_{\text{R}}^\pm\qty\Big(\varepsilon\Re{\dot{\zeta}_-(0)}, \varepsilon\Im{\dot{\zeta}_-(0)}, \varepsilon), \\
				\Lambda_2 (\nu_2, \varepsilon) &\coloneqq c^\sharp + \nu_2 - i\Im{\dot{\zeta}_-(0)}\cdot h_{\text{I}}^\pm\qty\Big(\varepsilon\Re{\dot{\zeta}_-(0)}, \varepsilon\Im{\dot{\zeta}_-(0)}, \varepsilon),
			\end{split}
		\end{equation}
		where $h_{\text{R}}^+$ and $h_\text{I}^+$ are functions given in \eqref{eqn c}. Now, consider the solution to ODE boundary value problem \eqref{eqn zeta} with parameter
		\begin{equation}\label{c thm}
			c \coloneqq \underbrace{\qty(c^{(k)}_+ + \nu_1)}_{\cR} + i\underbrace{\varepsilon\qty(c^\sharp + \nu_2)}_{\cI}.
		\end{equation}
		It is clear that, the dispersive relation \eqref{dispersion} is satisfied iff $\Lambda_1 = \Lambda_2 = 0$. Thus, it only remains to show that the map $\Uplambda \coloneqq (\Lambda_1, \Lambda_2)$ has a zero point near $(\nu_1, \nu_2) = (0, 0)$.
		
		Since $\Lambda_2$ is smooth on the region $\nu_2 > -c^\sharp$, Lemma \ref{lem 4} implies that the boundary value problem \eqref{eqn zeta} is uniquely solvable for sufficiently small $\nu_1$ and $\varepsilon$. Moreover, there holds
		\begin{equation*}
			\dot{\zeta}_-(0) = \xi_2(0) + i \Phi(0).
		\end{equation*}
		
		Now, consider the problem \eqref{limit ode} and \eqref{jump condition} with parameter $\cR$ given by \eqref{c thm}. As long as $\abs{\nu_1} \ll 1$, Lemma \ref{lem 4} yields the existence of the unique solution $(\wt{\vXi}, \wt{\Phi})$. Moreover, for a fixed constant $0 < \mu < 1$, there holds
		\begin{equation*}
			\dot{\zeta}_-(0) = \xi_2(0) + i\Phi(0) = \wt{\xi}_2(0) + i\wt{\Phi}(0) + \order{\varepsilon^\mu}.
		\end{equation*}
		Due to the $C^1$-dependence of $(\wt{\vXi}, \wt{\Phi})$ on $\nu_1$, it follows that
		\begin{equation*}
			\begin{split}
				\dot{\zeta}_-(0) &= \wt{\xi}_2(0) + i\wt{\Phi}(0) + \order{\varepsilon^\mu} \\
				&= \wh{\xi}^\sharp_2(0) + i\wh{\Phi}^\sharp(0) + \order{\varepsilon^\mu} + \order{\abs{\nu_1}} \\
				&= \wh{\xi}^\sharp_2(0) - i\pi\sum_{1\le j \le n}  \frac{\dot{\varpi}_-(s_j) \wh{\xi}^\sharp_1(s_j)}{\abs{\dot{w_-}(s_j)}}  + \order{\varepsilon^\mu} + \order{\abs{\nu_1}}.
			\end{split}
		\end{equation*}
		
		On the other hand, \eqref{h0} and the analyticity of $h_\text{R, I}^+$ yield that
		\begin{equation*}
			\Lambda_1 = \nu_1 + \order{\varepsilon} \qand \Lambda_2 = \nu_2 + \order{\varepsilon^\mu} + \order{\abs{\nu_1}}.
		\end{equation*}
		Thus, for each fixed $\varepsilon \ll 1$, $\Uplambda = \Uplambda(\nu_1, \nu_2, \varepsilon)$ admits a zero point near $(\nu_1, \nu_2) = (0, 0)$, which concludes the proof. The arguments for $c^{(k)}_-$ are the same.
	\end{proof}
	
	\subsection{Instability for Non-smooth Wind Profiles}\label{sec other instability}
	
	Finally, we present an example, for which the critical layer is away from the support of $\dot{\varpi}_-$. The construction is motivated from \cite[\S5]{BuhlerSWZ2016}.
	
	\begin{ex}[Constant inner vortices and piecewise-constant outer vortices]
		For simplicity, suppose now that $\cU_*^+$ is the unit disk and $\cU^-_* = \R^2\setminus\overline{\cU}_*^+$ (i.e., $\Rin = 0$ and $\Rout = \infty$). Assume further that background profiles are given as
		\begin{equation*}
			w_+(s) = B \qfor 0\le s \le 1, \qand w_-(s) = \begin{cases*}
				\omega_*[1-\exp(-2s)] + b\exp(-2s) &for $0 \le s < s_*$, \\
				\qty{\omega_*[\exp(2s_*)-1] + b}\exp(-2s) &for $s \ge s_*$,
			\end{cases*}
		\end{equation*}
		where $B$, $b$, $\omega_*$, and $s_*$ are all fixed constants, whose values/ranges will be determined later. 
		
		First, it is clear that
		\begin{equation*}
			\varpi_+(s) \equiv 2B \qfor 0<s < 1, \qand
			\varpi_-(s) = \begin{cases*}
				2\omega_* &for $0 < s < s_*$, \\
				0 &for $s>s_*$.
			\end{cases*}
		\end{equation*}
		Particularly, there holds
		\begin{equation*}
			\dv{s}\varpi_- = -2\omega_*\vdelta_{s_*},
		\end{equation*}
		where $\vdelta_{s_*}$ is the Dirac mass centered at $s_*$. It is easy to solve \eqref{BP} for $\zeta_+$ that
		\begin{equation*}
			\zeta_+(s) = \exp(\abs{k}s) \implies \dot{\zeta}_+(0) = \ak.
		\end{equation*}
		Similarly, since $\dot{\varpi}_-$ is a Dirac measure, the solution $\zeta_-$ can be given as
		\begin{equation*}
			\zeta_-(s) = \begin{cases*}
				A_1 \exp(\ak s) + A_2 \exp(-\ak s) &for $0 < s < s_*$,\\
				A_3 \exp(-\ak s) &for $s > s_*$,
			\end{cases*}
		\end{equation*}
		where $A_j \; (1 \le j \le 3)$ are constants so that
		\begin{equation*}
			\zeta_-(0) = 1 \qc \zeta_-(s_*-0) = \zeta_-(s_* + 0), \qand \dot{\zeta}_-(s_*+0) - \dot{\zeta}_-(s_*-0) = \frac{-2\omega_* \zeta_-(s_*)}{w_-(s_*) - c}.
		\end{equation*}
		Namely, $A_j \; (1 \le j \le 3)$ solve the linear algebraic equations:
		\begin{align*}
			A_1 + A_2 &= 1, \\
			A_1 \exp(\ak s_*) + A_2 \exp(-\ak s_*) &= A_3 \exp(-\ak s_*), \\
			\ak A_1 \exp(\ak s_*) - \ak A_2 \exp(-\ak s_*) &= -\ak A_3 \exp(-\ak s_*) + \frac{2\omega_*}{w_-(s_*) - c} A_3 \exp(-\ak s_*).
		\end{align*}
		Thus, it routine to calculate that
		\begin{equation*}
			\begin{split}
				\dot{\zeta}_-(0) &= \ak (A_1 - A_2) = - \ak \cdot \frac{c-\qty[w_-(s_*) - \frac{\omega_*}{\ak}\qty(1+e^{-2\ak s_*})]}{c - \qty[w_-(s_*) - \frac{\omega_*}{\ak}\qty(1-e^{-2\ak s_*})]} \\
				&= - \ak \cdot \frac{c-\qty[\omega_*\qty(1-e^{-2s_*}) + be^{-2s_*} - \frac{\omega_*}{\ak}\qty(1+e^{-2\ak s_*})]}{c - \qty[\omega_*\qty(1-e^{-2s_*}) + be^{-2s_*} - \frac{\omega_*}{\ak}\qty(1-e^{-2\ak s_*})]} \\
				&\eqqcolon -\ak \cdot \frac{c-\gamma_1}{c-\gamma_2}.
			\end{split}
		\end{equation*}
		Specifically, the dispersive relation \eqref{dispersion} now reads that
		\begin{equation}\label{ex disp}
			\begin{split}
				&\ak \qty[c-\qty(1-\frac{1}{\ak})B]^2 - \qty(1-\frac{1}{\ak})\qty[\frac{\alpha}{\varrho_+}\ak(\ak+1) - B^2] + \\
				&\quad + \varepsilon \qty[\ak\cdot \frac{c-\gamma_1}{c-\gamma_2}(c-b)^2 - 2\omega_*(c-b) - b^2] = 0.
			\end{split}
		\end{equation}
		Let $\ak \ge 2$ be fixed and the constant $B \ge 0$ satisfy
		\begin{equation*}
			B^2 < \frac{\alpha}{\varrho_+}\ak(\ak+1).
		\end{equation*}
		Then, when $\varepsilon = 0$, the algebraic equation \eqref{ex disp} for $c$ admits two distinct real roots:
		\begin{equation}
			\lambda_\pm^{(k)} = \qty(1-\frac{1}{\ak})B \pm \sqrt{\frac{\ak - 1}{k^2}\qty[\frac{\alpha}{\varrho_+}\ak(\ak+1) - B^2]}.
		\end{equation}
		Note that
		\begin{equation*}
			\begin{split}
				\gamma_2 &= \omega_*\qty(1-e^{-2s_*}) + be^{-2s_*} - \frac{\omega_*}{\ak}\qty(1-e^{-2\ak s_*}) \\
				&= \qty(1-\frac{1}{\ak})\omega_*\qty(1-e^{-2s_*}) + \qty[be^{-2s_*}- \frac{\omega_*}{\ak}\qty(e^{-2s_*} - e^{-2\ak s_*})].
			\end{split}
		\end{equation*}
		Thus, for well-chosen $\omega_*$ and $b$ depending on $\ak$ and $\lambda_+^{(k)}$, one can take a fixed position $s_* > 0$ so that
		\begin{equation}
			\gamma_2 = \lambda_+^{(k)}.
		\end{equation}
		Whence, by defining a function
		\begin{equation*}
			\begin{split}
				F(c, \varepsilon) &\coloneqq \qty(c-\lambda^{(k)}_-)\qty(c-\lambda^{(k)}_+)^2 + \\
				&\qquad + \varepsilon\qty[(c-\gamma_1)(c-b)^2 - 2\ak^{-1}\omega_*(c-b)\qty(c-\lambda^{(k)}_+) - \ak^{-1}b^2\qty(c-\lambda^{(k)}_+)],
			\end{split}
		\end{equation*}
		the dispersive relation \eqref{ex disp} can be rewritten as
		\begin{equation}\label{ex disp 2}
			F(c, \varepsilon) = 0.
		\end{equation}
		It is routine to check that
		\begin{gather*}
			F\qty(\lambda_+^{(k)}, 0) = (\pd_c F)\qty(\lambda_+^{(k)}, 0) = 0, \\
			(\pd_c\pd_c F)\qty(\lambda_+^{(k)}, 0) = 4 \qty(\lambda^{(k)}_+ - \lambda^{(k)}_-) > 0, \\
			(\pd_\varepsilon F)\qty(\lambda_+^{(k)}, 0) = (\gamma_2 - \gamma_1)\qty(\lambda_+^{(k)}-b)^2 > 0, \qq{whenever} b \neq \lambda_+^{(k)}.
		\end{gather*}
		Therefore, for each fixed $\varepsilon \ll 1$, the algebraic equation \eqref{ex disp 2} for $c$ admit two conjugate non-real roots, say, $\lambda_{\text{R}} \pm i \lambda_{\text{I}}$, for which $\lambda_{\text{R}}, \lambda_{\text{I}} \in \R$, $\lambda_{\text{I}} > 0$, and there hold
		\begin{equation*}
			\lambda_{\text{R}} = \lambda^{(k)}_+ + \order{\varepsilon^{\frac{1}{2}}} \qc \lambda_{\text{I}} = \order{\varepsilon^{\frac{1}{2}}}.
		\end{equation*}
		On the other hand, it follows that
		\begin{equation*}
			w_-(s_*) = \omega_*\qty(1-e^{-2s_*}) + be^{-2s_*} = \lambda_+^{(k)} + \frac{\omega_*}{\ak}\qty(1-e^{-2\ak s_*}),
		\end{equation*}
		which indicates that the critical layer is away form $\spt(\dot{\varpi}_-) = \{s=s_*\}$. Here we remark that the wind profile $w_-$ is piecewise smooth but only globally Lipschitz. Namely, the regularity of wind profile is crucial for relations among the instability, critical layers, and $\spt(\dot{\varpi}_-)$.
	\end{ex}

	\section*{Acknowledgments}
	The research of Changfeng Gui is supported by 
	University of Macau research grants CPG2024-00016-FST, CPG2025-00032-FST, SRG2023-00011-FST, MYRGGRG2023-00139-FST-UMDF, UMDF Professorial Fellowship of Mathematics, Macao SAR FDCT 0003/2023/RIA1 and Macao SAR FDCT 0024/2023/RIB1. The research of Sicheng Liu  is supported by the UM Postdoctoral Fellow (UMPF) scheme under the UM Talent Programme at the University of Macau.
	
	
	\subsection*{Data availability}
	This manuscript has no associated data.
	
	\subsection*{Conflict of interest}
	The authors have no conflict of interest to disclose.
	
	\fancyhead[RO,LE]{\sc{\sffamily References}}
	\bibliographystyle{amsplain0}
	{\small\bibliography{ref}}
	
	\end{document}

%% file: mathsymb.tex
\theoremstyle{definition}
\newtheorem{defi}{\sffamily Definition}[section]
\newtheorem{ex}[defi]{\sffamily Example}
\theoremstyle{plain}
\newtheorem{theorem}[defi]{\sffamily Theorem}
\newtheorem{prop}[defi]{\sffamily Proposition}
\newtheorem{lemma}[defi]{\sffamily Lemma}

\theoremstyle{remark}
\newtheorem*{remark}{\sffamily Remark}

\makeatletter
\providecommand{\proofnamestyle}{\itshape\sffamily\bfseries}
\renewenvironment{proof}[1][\proofname]{\par
	\pushQED{\qed}%
	\normalfont \topsep6\p@\@plus6\p@\relax
	\trivlist
	\item\relax
	{\proofnamestyle
		#1\@addpunct{.}}\hspace\labelsep\ignorespaces
}{%
	\par \popQED\endtrivlist\@endpefalse
}
\makeatother

\numberwithin{equation}{section}

\newcommand*{\wt}[1]{\widetilde{#1}}
\newcommand*{\wh}[1]{\widehat{#1}}
\newcommand*{\R}{\mathbb{R}}
\newcommand*{\I}{\mathbbm{1}}

\providecommand\given{} 
\newcommand\SetSymbol[1][]{\nonscript\:#1\vert \allowbreak \nonscript\: \mathopen{}}
\DeclarePairedDelimiterX\Set[1]\{\}{ \renewcommand\given{\SetSymbol[\delimsize]} #1 }

\DeclareMathOperator*{\sgn}{\operatorname{sgn}}
\DeclareMathOperator*{\spt}{\operatorname{spt}}

\DeclareMathOperator*{\Curl}{\operatorname{curl}}

\newcommand*{\Dt}{\mathop{}\!\textbf{D}_t}
\newcommand*{\Dbt}{\mathop{}\!\mathcal{D}_\beta}

\newcommand*{\pd}{\partial}

\newcommand*{\Z}{\mathbb{Z}}
\newcommand*{\C}{\mathbb{C}}

\newcommand*{\cR}{c_{\text{R}}}
\newcommand*{\cI}{c_{\text{I}}}

\newcommand*{\vv}{{\vb{v}}}

\newcommand*{\cU}{\mathcal{U}}
\newcommand*{\cUs}{\mathcal{U}_*}
\newcommand*{\Gmt}{\Gamma_t}
\newcommand*{\Gms}{\Gamma_*}
\newcommand*{\vn}{{\vb{n}}}
\newcommand*{\cV}{\mathcal{V}}
\newcommand*{\Rin}{R_{\text{in}}}
\newcommand*{\Rout}{R_{\text{out}}}
\newcommand*{\er}{{\vu{e}_r}}
\newcommand*{\etheta}{{\vu{e}_\theta}}
\newcommand*{\vV}{\vb{V}}

\newcommand*{\vbu}{{\vb{u}}}
\newcommand*{\vpsi}{{\vb*{\uppsi}}}
\newcommand*{\vtau}{{\vb*{\uptau}}}
\newcommand*{\vXi}{{\vb*{\Upxi}}}

\newcommand*{\vdelta}{\vb*{\updelta}}

\newcommand*{\ak}{\abs{k}}